\documentclass[12pt]{amsart}
\usepackage{amsmath,amsthm,amsfonts,amssymb,mathrsfs}
\usepackage{color}
\date{\today}

\usepackage{hyperref}
\input{xy}
 \xyoption{all}
 \xyoption{arc}

\newtheorem{theorem}{Theorem}

\newtheorem{proposition}{Proposition}
\newtheorem{corollary}{Corollary}

\newtheorem{lemma}{Lemma}
\theoremstyle{definition}

\newtheorem{example}{Example}
\newtheorem{remark}{Remark}

\begin{document}

\title[On feebly compact topologies on the semilattice $\exp_n\lambda$]{On feebly compact topologies on the semilattice $\exp_n\lambda$}
\author{Oleg~Gutik and Oleksandra Sobol}
\address{Department of Mechanics and Mathematics,  
National University of Lviv, Universytetska 1, Lviv, 79000, Ukraine}
\email{o\_\,gutik@franko.lviv.ua, ovgutik@yahoo.com, olesyasobol@mail.ru}



\begin{abstract}
We study feebly compact topologies $\tau$ on the semilattice $\left(\exp_n\lambda,\cap\right)$ such that $\left(\exp_n\lambda,\tau\right)$ is a semitopological semilattice. All compact semilattice $T_1$-topologies on $\exp_n\lambda$ are described. Also we prove that for an arbitrary positive integer $n$ and an arbitrary infinite cardinal $\lambda$ for a $T_1$-topology $\tau$ on $\exp_n\lambda$ the following conditions are equivalent: $(i)$ $\left(\exp_n\lambda,\tau\right)$ is a compact topological semilattice; $(ii)$ $\left(\exp_n\lambda,\tau\right)$ is a countably compact topological semilattice; $(iii)$ $\left(\exp_n\lambda,\tau\right)$ is a feebly compact topological semilattice; $(iv)$ $\left(\exp_n\lambda,\tau\right)$ is a compact semitopological semilattice; $(v)$ $\left(\exp_n\lambda,\tau\right)$ is a countably compact semitopological semilattice. We construct a countably pracompact $H$-closed quasiregular non-semiregular topology $\tau_{\operatorname{\textsf{fc}}}^2$ such that $\left(\exp_2\lambda,\tau_{\operatorname{\textsf{fc}}}^2\right)$ is a semitopological semilattice with discontinuous semilattice operation and prove that for an arbitrary positive integer $n$ and an arbitrary infinite cardinal $\lambda$ every $T_1$-semiregular feebly compact semitopological semilattice $\exp_n\lambda$ is a compact topological semilattice.
\end{abstract}

\maketitle


\begin{center}
  \emph{\textbf{Dedicated to the memory of Professor Mykola Komarnyts'kyy}}
\end{center}

\bigskip


We shall follow the terminology of~\cite{Carruth-Hildebrant-Koch-1983-1986, Clifford-Preston-1961-1967,
Engelking-1989, Gierz-Hofmann-Keimel-Lawson-Mislove-Scott-2003, Ruppert-1984}. If $X$ is a topological space and $A\subseteq X$, then by $\operatorname{cl}_X(A)$ and $\operatorname{int}_X(A)$ we denote the topological closure and interior of $A$ in $X$, respectively. By $\omega$ we denote the first infinite cardinal.

A semigroup $S$ is called an \emph{inverse semigroup} if every $a$
in $S$ possesses an unique inverse, i.e. if there exists an unique
element $a^{-1}$ in $S$ such that
\begin{equation*}
    aa^{-1}a=a \qquad \mbox{and} \qquad a^{-1}aa^{-1}=a^{-1}.
\end{equation*}
A map which associates to any element of an inverse semigroup its
inverse is called the \emph{inversion}.

A {\it topological} ({\it inverse}) {\it semigroup} is a topological space together with a continuous semigroup operation (and an~inversion, respectively). Obviously, the inversion defined on a topological inverse semigroup is a homeomorphism. If $S$ is a~semigroup (an~inverse semigroup) and $\tau$ is a topology on $S$ such that $(S,\tau)$ is a topological (inverse) semigroup, then we
shall call $\tau$ a \emph{semigroup} (\emph{inverse}) \emph{topology} on $S$. A {\it semitopological semigroup} is a topological space together with a separately continuous semigroup operation.

If $S$ is a~semigroup, then by $E(S)$ we denote the
subset of all idempotents of $S$. On  the set of idempotents
$E(S)$ there exists a natural partial order: $e\leqslant f$
\emph{if and only if} $ef=fe=e$. A \emph{semilattice} is a commutative semigroup of idempotents. A {\it topological} ({\it semitopological}) {\it semilattice} is a topological space together with a continuous (separately continuous) semilattice operation. If $S$ is a~semilattice and $\tau$ is a topology on $S$ such that $(S,\tau)$ is a topological semilattice, then we shall call $\tau$ a \emph{semilattice} \emph{topology} on $S$.

Let $\lambda$ be an arbitrary non-zero cardinal. A map $\alpha$ from a subset $D$ of $\lambda$ into $\lambda$ is called a \emph{partial transformation} of $X$. In this case the set $D$ is called the \emph{domain} of $\alpha$ and it is denoted by $\operatorname{dom}\alpha$. The image of an element $x\in\operatorname{dom}\alpha$ under $\alpha$ we shall denote by $x\alpha$  Also, the set $\{ x\in X\colon y\alpha=x \mbox{ for some } y\in Y\}$ is called the \emph{range} of $\alpha$ and is denoted by $\operatorname{ran}\alpha$. The cardinality of $\operatorname{ran}\alpha$ is called the \emph{rank} of $\alpha$ and denoted by $\operatorname{rank}\alpha$. For convenience we denote by $\varnothing$ the empty transformation, that is a partial mapping
with $\operatorname{dom}\varnothing=\operatorname{ran}\varnothing=\varnothing$.

Let $\mathscr{I}_\lambda$ denote the set of all partial one-to-one transformations of $\lambda$ together with the following semigroup operation:
\begin{equation*}
    x(\alpha\beta)=(x\alpha)\beta \quad \mbox{if} \quad
    x\in\operatorname{dom}(\alpha\beta)=\{
    y\in\operatorname{dom}\alpha\colon
    y\alpha\in\operatorname{dom}\beta\}, \qquad \mbox{for} \quad
    \alpha,\beta\in\mathscr{I}_\lambda.
\end{equation*}
The semigroup $\mathscr{I}_\lambda$ is called the \emph{symmetric
inverse semigroup} over the cardinal $\lambda$~(see \cite{Clifford-Preston-1961-1967}). The symmetric
inverse semigroup was introduced by V.~V.~Wagner~\cite{Wagner-1952}
and it plays a major role in the theory of semigroups.

Put
$\mathscr{I}_\lambda^n=\{ \alpha\in\mathscr{I}_\lambda\colon
\operatorname{rank}\alpha\leqslant n\}$,
for $n=1,2,3,\ldots$. Obviously,
$\mathscr{I}_\lambda^n$ ($n=1,2,3,\ldots$) are inverse semigroups,
$\mathscr{I}_\lambda^n$ is an ideal of $\mathscr{I}_\lambda$, for each $n=1,2,3,\ldots$. The semigroup
$\mathscr{I}_\lambda^n$ is called the \emph{symmetric inverse semigroup of
finite transformations of the rank $\leqslant n$} \cite{Gutik-Lawson-Repov-2009, Gutik-Reiter-2009}. The empty partial map $\varnothing\colon \lambda\rightharpoonup\lambda$ we denote by $0$. It is obvious that $0$ is zero of the semigroup $\mathscr{I}_\lambda^n$.

Let $\lambda$ be a non-zero cardinal. On the set
 $
 B_{\lambda}=(\lambda\times\lambda)\cup\{ 0\}
 $,
where $0\notin\lambda\times\lambda$, we define the semigroup
operation ``$\, \cdot\, $'' as follows
\begin{equation*}
(a, b)\cdot(c, d)=
\left\{
  \begin{array}{cl}
    (a, d), & \hbox{ if~ } b=c;\\
    0, & \hbox{ if~ } b\neq c,
  \end{array}
\right.
\end{equation*}
and $(a, b)\cdot 0=0\cdot(a, b)=0\cdot 0=0$ for $a,b,c,d\in
\lambda$. The semigroup $B_{\lambda}$ is called the
\emph{semigroup of $\lambda\times\lambda$-matrix units}~(see
\cite{Clifford-Preston-1961-1967}). Obviously, for any cardinal $\lambda>0$, the semigroup
of $\lambda\times\lambda$-matrix units $B_{\lambda}$ is isomorphic
to $\mathscr{I}_\lambda^1$.

A subset $A$ of a topological space $X$ is called \emph{regular open} if $\operatorname{int}_X(\operatorname{cl}_X(A))=A$.

We recall that a topological space $X$ is said to be
\begin{itemize}
  \item \emph{functionally Hausdorff} if for every pair of distinct points $x_1,x_2\in X$ there exists a continuous function $f\colon X\rightarrow [0,1]$ such that $f(x_1)=0$ and $f(x_2)=1$;
  \item \emph{semiregular} if $X$ has a base consisting of regular open subsets;
  \item \emph{quasiregular} if for any non-empty open set $U\subset X$ there exists a non-empty open set $V\subset U$ such that $\operatorname{cl}_X(V) \subseteq U$;
  \item \emph{compact} if each open cover of $X$ has a finite subcover;
  \item \emph{sequentially compact} if each sequence $\{x_i\}_{i\in\mathbb{N}}$ of $X$ has a convergent subsequence in $X$;
  \item \emph{countably compact} if each open countable cover of $X$ has a finite subcover;
  \item \emph{$H$-closed} if $X$ is a closed subspace of every Hausdorff topological space in which it contained;
  \item \emph{countably compact at a subset} $A\subseteq X$ if every infinite subset $B\subseteq A$  has  an  accumulation  point $x$ in $X$;
  \item \emph{countably pracompact} if there exists a dense subset $A$ in $X$  such that $X$ is countably compact at $A$;
  \item \emph{feebly compact} if each locally finite open cover of $X$ is finite;
  \item \emph{pseudocompact} if $X$ is Tychonoff and each continuous real-valued function on $X$ is bounded.
\end{itemize}
According to Theorem~3.10.22 of \cite{Engelking-1989}, a Tychonoff topological space $X$ is feebly compact if and only if $X$ is pseudocompact. Also, a Hausdorff topological space $X$ is feebly compact if and only if every locally finite family of non-empty open subsets of $X$ is finite.  Every compact space and every sequentially compact space are countably compact, every countably compact space is countably pracompact, and every countably pracompact space is feebly compact (see \cite{Arkhangelskii-1992}), and every $H$-closed space is feebly compact too (see \cite{Gutik-Ravsky-2015a}).

Topological properties of an infinite (semi)topological semigroup $\lambda\times \lambda$-matrix units studied in \cite{Gutik-Pavlyk-2005, Gutik-Pavlyk-2005a, Gutik-Pavlyk-Reiter-2009}. In \cite{Gutik-Pavlyk-2005a} showed that on the infinite semitopological semigroup $\lambda\times \lambda$-matrix units $B_\lambda$ there exists a unique Hausdorff topology $\tau_c$ such that $(B_\lambda,\tau_c)$ is a compact semitopological semigroup and there proved that every pseudocompact Hausdorff topology $\tau$ on $B_\lambda$ such that $(B_\lambda,\tau_c)$ is a semitopological semigroup, is compact. Also, in \cite{Gutik-Pavlyk-2005a} proved that every non-zero element of a Hausdorff semitopological semigroup $\lambda\times \lambda$-matrix units $B_\lambda$ is an isolated point in the topological space $B_\lambda$. In \cite{Gutik-Pavlyk-2005} proved that infinite semigroup $\lambda\times \lambda$-matrix units $B_\lambda$ does not embed into a compact Hausdorff topological semigroup, every Hausdorff topological inverse semigroup $S$ such that contains $B_\lambda$ as a subsemigroup, contains $B_\lambda$ as a closed subsemigroup, i.e., $B_\lambda$ is \emph{algebraically complete} in the class of Hausdorff topological inverse semigroups. This result in \cite{Gutik-Lawson-Repov-2009} was extended onto so called inverse semigroups with \emph{tight ideal series} and as a corollary onto the semigroup $\mathscr{I}_\lambda^n$. Also, in \cite{Gutik-Reiter-2009} there proved that for every positive integer $n$ the semigroup $\mathscr{I}_\lambda^n$ is \emph{algebraically $h$-complete} in the class of Hausdorff topological inverse semigroups, i.e., every homomorphic image of $\mathscr{I}_\lambda^n$ is algebraically complete in the class of Hausdorff topological inverse semigroups. In the paper \cite{Gutik-Reiter-2010} this result was extended onto the class of Hausdorff semitopological inverse semigroups and there it was shown that for an infinite cardinal $\lambda$ the semigroup $\mathscr{I}_\lambda^n$ admits a unique Hausdorff topology $\tau_c$ such that $(\mathscr{I}_\lambda^n,\tau_c)$ is a compact semitopological semigroup. Also there proved the every countably compact Hausdorff topology $\tau$ on $B_\lambda$ such that $(\mathscr{I}_\lambda^n,\tau_c)$ is a semitopological semigroup is compact. In \cite{Gutik-Pavlyk-Reiter-2009} it was shown that a topological semigroup of finite partial bijections $\mathscr{I}_\lambda^n$ with a compact subsemigroup of idempotents is absolutely $H$-closed (i.e., every homomorphic image of $\mathscr{I}_\lambda^n$ is algebraically complete in the class of Hausdorff topological semigroups) and any
countably compact topological semigroup does not contain $\mathscr{I}_\lambda^n$ as a subsemigroup for infinite $\lambda$. In \cite{Gutik-Pavlyk-Reiter-2009} it was given sufficient conditions onto a topological semigroup $\mathscr{I}_\lambda^1$ to be non-$H$-closed. Also in \cite{Gutik-2014} it was proved that an infinite semitopological semigroup of $\lambda\times\lambda$-matrix units $B_\lambda$ is $H$-closed in the class of semitopological semigroups if and only if the space $B_\lambda$ is compact.

For an arbitrary positive integer $n$ and an arbitrary non-zero cardinal $\lambda$ we put
\begin{equation*}
  \exp_n\lambda=\left\{A\subseteq \lambda\colon |A|\leqslant n\right\}.
\end{equation*}

It is obvious that for any positive integer $n$ and any non-zero cardinal $\lambda$ the set $\exp_n\lambda$ with the binary operation $\cap$ is a semilattice. Later in this paper by $\exp_n\lambda$ we shall denote the semilattice $\left(\exp_n\lambda,\cap\right)$. It is easy to see that $\exp_n\lambda$ is isomorphic to the subsemigroup of idempotents (the band) of the semigroup $\mathscr{I}_\lambda^n$ for any positive integer $n$.

In this paper we study feebly compact topologies $\tau$ on the semilattice $\exp_n\lambda$ such that $\left(\exp_n\lambda,\tau\right)$ is a semitopological semilattice. All compact semilattice $T_1$-topologies on $\exp_n\lambda$ are described. We prove that for an arbitrary positive integer $n$ and an arbitrary infinite cardinal $\lambda$ every $T_1$-semitopological countably compact semilattice $\left(\exp_n\lambda,\tau\right)$ is a compact topological semilattice. Also,  we construct a countably pracompact $H$-closed quasiregular non-semiregular topology $\tau_{\operatorname{\textsf{fc}}}^2$ such that $\left(\exp_2\lambda,\tau_{\operatorname{\textsf{fc}}}^2\right)$ is a semitopological semilattice with the discontinuous semilattice operation and show that for an arbitrary positive integer $n$ and an arbitrary infinite cardinal $\lambda$ a semiregular feebly compact semitopological semilattice $\exp_n\lambda$ is a compact topological semilattice.



\bigskip

We recall that a topological space $X$ is said to be
\begin{itemize}
  \item \emph{scattered} if $X$ does not contain a non-empty dense-in-itself subspace;
  \item \emph{hereditarily disconnected} (or \emph{totally disconnected}) if $X$ does not contain any connected subsets of cardinality larger than one.
\end{itemize}

\begin{proposition}\label{proposition-2.1}
Let $n$ be an arbitrary positive integer and $\lambda$ be an arbitrary infinite cardinal. Then for every $T_1$-topology $\tau$ on $\exp_n\lambda$ such that $\left(\exp_n\lambda,\tau\right)$ is a semitopological semilattice the following assertions hold:
\begin{itemize}
  \item[$(i)$] $\left(\exp_n\lambda,\tau\right)$ is a closed subset of any $T_1$-semitopological semilattice $S$ which contains $\exp_n\lambda$ as a subsemilattice;
  \item[$(ii)$] for every $x\in \exp_n\lambda$ there exists an open neighbourhood $U(x)$ of the point $x$ in the space $\left(\exp_n\lambda,\tau\right)$ such that $U(x)\subseteq{\uparrow}x$;
  \item[$(iii)$] ${\uparrow}x$ is a open-and-closed subset of the space $\left(\exp_n\lambda,\tau\right)$ for every $x\in \exp_n\lambda$;
   \item[$(iv)$] the topological space $\left(\exp_n\lambda,\tau\right)$ is functionally Hausdorff and quasiregular, and hence is Hausdorff;
   \item[$(v)$] $\left(\exp_n\lambda,\tau\right)$ is a scattered hereditarily disconnected space.
\end{itemize}
\end{proposition}

\begin{proof}
$(i)$ We shall prove our assertion by induction.

Let $n=1$ and let $S$ be an arbitrary $T_1$-se\-mi\-to\-po\-lo\-gi\-cal semilattice which contains $\exp_1\lambda$ as a proper subsemilattice. We fix an arbitrary element $x\in S\setminus\exp_1\lambda$. Suppose to the contrary that every open neighbourhood $U(x)$ of the point $x$ in the topological space $S$ intersects the semilattice $\exp_1\lambda$. First we shall show that $ex=0$ for any $e\in\exp_1\lambda$, where $0$ is zero of the semilattice $\exp_1\lambda$. Suppose to the contrary that there exists $e\in\exp_1\lambda$ such that $ex=y\neq 0$. Since $S$ is a $T_1$-space there exists an open neighbourhood $U(y)$ of the point $y$ in $S$ such that $0\notin U(y)$. Then by the definition the semilattice operation of $\exp_1\lambda$ and by separate continuity of the semilattice operation of $S$ we have that $0\in e\cdot V(x)\subseteq U(y)$ for every open neighbourhood $V(x)$ of the point $x$ in $S$, because the neighbourhood $V(x)$ contains infinitely many points from the semilattice $\exp_1\lambda$. This contradicts the choice of the neighbourhood $U(y)$. The obtained contradiction implies that $ex=0$ for any $e\in\exp_1\lambda$. Fix an arbitrary open neighbourhood $U(x)$ of $x$ in $S$ such that $0\notin U(x)$. Then by the separate continuity of the semilattice operation of $S$ we get that there exists an open neighbourhood $V(x)$ of $x$ in $S$ such that $x\cdot V(x)\subseteq U(x)$. Since $V(x)$ intersects the semilattice $\exp_1\lambda$, the above arguments imply that $0\in x\cdot V(x)$, a contradiction. Therefore, $\exp_1\lambda$ is a closed sub\-semi\-lattice of $S$.

Suppose that for every $j<k$ the semilattice $\exp_j\lambda$ is a closed subsemilattice of any $T_1$-se\-mi\-to\-po\-lo\-gi\-cal semilattice which contains $\exp_j\lambda$ as a proper subsemilattice, where $k\leqslant n$. We shall prove that this implies that $\exp_k\lambda$ is a closed subsemilattice of any $T_1$-se\-mi\-to\-po\-lo\-gi\-cal semilattice which contains $\exp_k\lambda$ as a proper subsemilattice. Suppose to the contrary that there exists a $T_1$-se\-mi\-to\-po\-lo\-gi\-cal semilattice $S$ which contains $\exp_k\lambda$ as a non-closed subsemilattice. Then there exists an element $x\in S\setminus\exp_k\lambda$ such that every open neighbourhood $U(x)$ of the point $x$ in the topological space $S$ intersects the semilattice $\exp_k\lambda$. The assumption of induction implies that there exists an open neighbourhood $U(x)$ of the point $x$ in $S$ such that $U(x)\cap \exp_k\lambda\subseteq \exp_k\lambda\setminus \exp_{k-1}\lambda$. Now, as in the case of the semilattice $\exp_1\lambda$ the separate continuity of the semilattice operation of $S$ implies that $e\cdot x\in \exp_{k-1}\lambda$ for any $e\in \exp_k\lambda\setminus \exp_{k-1}\lambda$. Indeed, suppose to the contrary that there exists $e\in \exp_k\lambda\setminus \exp_{k-1}\lambda$ such that $e\cdot x=z\notin \exp_{k-1}\lambda$. Then the assumption of induction implies that $\exp_{k-1}\lambda$ is a closed subsemilattice of $S$ and hence there exists an open neighbourhood $U(y)$ of the point $y$ in $S$ such that $U(y)\cap\exp_{k-1}\lambda=\varnothing$. Now, by the separate continuity of the semilattice operation of $S$ there exists an open neighbourhood $U(x)$ of the point $x$ in $S$ such that $e\cdot U(x)\subseteq U(y)$. Then the semilattice operation of $\exp_k\lambda$ implies that $\left(e\cdot U(x)\right)\cap\exp_{k-1}\lambda\neq\varnothing$, which contradicts the choice of the neighbourhood $U(y)$.

Fix an arbitrary open neighbourhood $U(x)$ of $x$ in $S$ such that $U(x)\cap \exp_k\lambda\subseteq \exp_k\lambda\setminus \exp_{k-1}\lambda$. Then by the separate continuity of the semilattice operation of $S$ we get that there exists an open neighbourhood $V(x)\subseteq U(x)$ of $x$ in $S$ such that $x\cdot V(x)\subseteq U(x)$. By our assumption we have that the set $V(x)\cap\exp_k\lambda\setminus \exp_{k-1}\lambda$ is infinite and hence the above part of our proof implies that $\left(x\cdot V(x)\right)\cap \exp_{k-1}\lambda\neq\varnothing$, which contradicts the choice of the neighbourhood $U(x)$. The obtained contradiction implies that $\exp_k\lambda$ is a closed subset of $S$, which completes the proof of our assertion.

$(ii)$ In the case when $x=0$ the statement is trivial, and hence we assume that $x\neq 0$. Then the definition of the semilattice $\exp_n\lambda$ implies that there exists the minimum positive integer $k$ such that $x\in\exp_k\lambda$ and $x\notin\exp_{k-1}\lambda$. By item $(i)$ there exists an open neighbourhood $U(x)$ of the point $x$ in the space $\left(\exp_n\lambda,\tau\right)$ such that $U(x)\subseteq \exp_n\lambda\setminus\exp_{k-1}\lambda$. Then the separate continuity of the semilattice operation in $\left(\exp_n\lambda,\tau\right)$ implies that there exists an open neighbourhood $V(x)\subseteq U(x)$ such that $x\cdot V(x)\subseteq U(x)$. If $V(x)\nsubseteq{\uparrow}x$ then by the definition of the semilattice operation on $\exp_n\lambda$ we have that there exists $y\in V(x)$ such that $xy\in \exp_{k-1}\lambda$, a contradiction. Hence we get that $V(x)\subseteq{\uparrow}x$.

$(iii)$ Since a topological space is $T_1$-space if and only if every its point is a closed subset of itself, the separate continuity of the semilattice operation implies that ${\uparrow}x$ is a closed subset of $\left(\exp_n\lambda,\tau\right)$ for any $x\in \exp_n\lambda$. Also, item $(ii)$ implies that
\begin{equation*}
  {\uparrow}x=\bigcup\left\{V(y)\colon y\in{\uparrow}x \hbox{~and~} V(y) \hbox{~is an open neighbourhood} \hbox{~of~} y \hbox{~such that~} V(y)\subseteq {\uparrow}y\right\}
\end{equation*}
is an open subset of $\left(\exp_n\lambda,\tau\right)$ for any $x\in \exp_n\lambda$.

$(iv)$ Fix arbitrary distinct elements $x_1$ and $x_2$ of the semitopological semilattice $\left(\exp_n\lambda,\tau\right)$. Then we have either $x_1\notin{\uparrow}x_2$ or $x_2\notin{\uparrow}x_1$. In the case when $x_1\notin{\uparrow}x_2$ we define the map $f\colon \left(\exp_n\lambda,\tau\right)\to[0,1]$ by the formula
\begin{equation*}
  f(x)=
  \left\{
    \begin{array}{ll}
      1, & \hbox{if~} x\in{\uparrow}x_2; \\
      0, & \hbox{if~} x\notin{\uparrow}x_2.
    \end{array}
  \right.
\end{equation*}
Then we have that $f(x_1)=0$ and $f(x_2)=1$ and by item $(iii)$ ${\uparrow}x_2$ is an open-and-closed subset of the space $\left(\exp_n\lambda,\tau\right)$, and hence so defined map $f\colon \left(\exp_n\lambda,\tau\right)\to[0,1]$ is continuous.

The definition of the semilattice $\exp_n\lambda$ implies that every non-empty open subset of $\left(\exp_n\lambda,\tau\right)$ has a maximal element $x$ with the respect to the natural partial order on $\exp_n\lambda$. Then by item $(iii)$, ${\uparrow}x$ is an open-and-closed subset of $\left(\exp_n\lambda,\tau\right)$, and hence $x$ is an isolated point of $\left(\exp_n\lambda,\tau\right)$. Since $\tau$ is a $T_1$-topology, $\operatorname{cl}_{\exp_n\lambda}(\{x\})=\{x\}\subseteq U$, which implies that $\left(\exp_n\lambda,\tau\right)$ is a quasiregular space.

$(v)$ We shall prove that every non-empty subset $A$ of $\left(\exp_n\lambda,\tau\right)$ has an isolated point in itself. Fix an arbitrary non-empty subset $A$ of $\left(\exp_n\lambda,\tau\right)$. If $A\cap\exp_n\lambda\setminus\exp_{n-1}\lambda\neq\varnothing$ then by item $(ii)$ every point $x\in A\cap\exp_n\lambda\setminus\exp_{n-1}\lambda$ is isolated in $\left(\exp_n\lambda,\tau\right)$ and hence $x$ is an isolated point of $A$. In the other case there exists a positive integer $k<n$ such that $A\subseteq \exp_k\lambda$ and $A\nsubseteq \exp_{k-1}\lambda$. Then by item $(ii)$ every point $x\in A\cap\exp_k\lambda\setminus\exp_{k-1}\lambda$ is isolated in $A$.

The hereditary disconnectedness of the space $\left(\exp_n\lambda,\tau\right)$ follows from item $(iii)$. Indeed, if $x \nleqslant y$ in $\exp_n\lambda$ then by item $(iii)$, ${\uparrow}x$ is an open-and-closed neighbourhood of $x$ in $\left(\exp_n\lambda,\tau\right)$ such that $y\notin{\uparrow}x$. This implies that the space $\left(\exp_n\lambda,\tau\right)$ does not contain any connected subsets of cardinality larger than one.
\end{proof}

Recall \cite{Gutik-2014} an algebraic semilattice $S$ is called \emph{algebraically complete} in the class $\mathfrak{STSL}$ of semitopological semilattices if $S$ is a closed subsemilattice of every semitopological semilattice $L\in\mathfrak{STSL}$ which contains $S$ as a subsemilattice.

Proposition~\ref{proposition-2.1}$(i)$ implies the following corollary.

\begin{corollary}\label{corollary-2.2}
Let $n$ be an arbitrary positive integer and $\lambda$ be an arbitrary infinite cardinal. Then the semilattice $\exp_n\lambda$ is algebraically complete in the class of $T_1$-semitopological semilattices.
\end{corollary}

The following example shows that the statement $(iv)$ of Proposition~\ref{proposition-2.1} does not hold in the case when $\left(\exp_n\lambda,\tau\right)$ is a $T_0$-space.

\begin{example}\label{example-2.3}
For an arbitrary positive integer $n$ and an arbitrary infinite cardinal $\lambda$ we define a topology $\tau_0$ on $\exp_n\lambda$ in the following way:
\begin{itemize}
  \item[$(i)$] all non-zero elements of the semilattice $\exp_n\lambda$ are isolated points in $\left(\exp_n\lambda,\tau_0\right)$; \; and
  \item[$(ii)$] $\exp_n\lambda$ is the unique open neighbourhood of zero in $\left(\exp_n\lambda,\tau_0\right)$.
\end{itemize}
Simple verifications show that the semilattice operation on $\left(\exp_n\lambda,\tau_0\right)$ is continuous.
\end{example}

\begin{example}\label{example-2.4}
For an arbitrary positive integer $n$ and an arbitrary infinite cardinal $\lambda$ we define a topology $\tau_{\operatorname{\textsf{c}}}^n$ on $\exp_n\lambda$ in the following way: the family $\left\{\mathscr{B}^n_{\operatorname{\textsf{c}}}(x)\colon x\in \exp_n\lambda\right\}$, where
\begin{equation*}
\mathscr{B}^n_{\operatorname{\textsf{c}}}(x)=\left\{U_x(x_1,\ldots,x_j)={\uparrow}x\setminus\left({\uparrow}x_1\cup\cdots\cup{\uparrow}x_j\right)\colon x_1,\ldots,x_j\in{\uparrow}x\setminus\{x\}\right\},
\end{equation*}
forms a neighbourhood system for the topological space $\left(\exp_n\lambda,\tau_{\operatorname{\textsf{c}}}^n\right)$. Simple verifications show that the family $\left\{\mathscr{B}^n_{\operatorname{\textsf{c}}}(x)\colon x\in \exp_n\lambda\right\}$ satisfies the properties \textbf{(BP1)--(BP3)} of \cite{Engelking-1989}. Also, it is obvious that the family $\left\{\mathscr{B}^n_{\operatorname{\textsf{c}}}(x)\colon x\in \exp_n\lambda\right\}$ satisfies the property \textbf{(BP4)} of \cite[Proposition~1.5.2]{Engelking-1989}, and hence the topological space $\left(\exp_n\lambda,\tau_{\operatorname{\textsf{c}}}^n\right)$ is Hausdorff.
\end{example}

Recall \cite{Engelking-1989} a topological space $X$ is called \emph{$0$-dimensional} if $X$ has a base which consists of open-and-closed subsets of $X$.

\begin{proposition}\label{proposition-2.5}
Let $n$ be an arbitrary positive integer and $\lambda$ be an arbitrary infinite cardinal. Then \emph{$\left(\exp_n\lambda,\tau_{\operatorname{\textsf{c}}}^n\right)$} is a compact $0$-dimensional topological semilattice.
\end{proposition}

\begin{proof}
The definition of the family $\left\{\mathscr{B}^n_{\operatorname{\textsf{c}}}(x)\colon x\in \exp_n\lambda\right\}$ implies that for arbitrary $x\in \exp_n\lambda$ the set ${\uparrow}x$ is open-and-closed in $\left(\exp_n\lambda,\tau_{\operatorname{\textsf{c}}}^n\right)$, and hence $\left\{\mathscr{B}^n_{\operatorname{\textsf{c}}}(x)\colon x\in \exp_n\lambda\right\}$ is the base of the topological space $\left(\exp_n\lambda,\tau_{\operatorname{\textsf{c}}}^n\right)$ which consists of open-and-closed subsets.

Now, by induction we shall show that the space $\left(\exp_n\lambda,\tau_{\operatorname{\textsf{c}}}^n\right)$ is compact. In the case when $n=1$ the compactness of $\left(\exp_1\lambda,\tau_{\operatorname{\textsf{c}}}^1\right)$ follows from the definition of the family $\left\{\mathscr{B}^1_{\operatorname{\textsf{c}}}(x)\colon x\in \exp_1\lambda\right\}$. Next, we shall prove that the statement the space $\left(\exp_i\lambda,\tau_{\operatorname{\textsf{c}}}^i\right)$ is compact for all positive integers $i<k\leqslant n$ implies that the space $\left(\exp_{k}\lambda,\tau_{\operatorname{\textsf{c}}}^{k}\right)$ is compact too. Fix an arbitrary open cover $\mathscr{U}$ of the topological space $\left(\exp_{k}\lambda,\tau_{\operatorname{\textsf{c}}}^{k}\right)$. The definition of the topology $\tau_{\operatorname{\textsf{c}}}^{k}$ implies that there exists an element $U_0\in\mathscr{U}$ such that $0\in U_0$. Then there exists $U_0(x_1,\ldots,x_j)\in\mathscr{B}^k_{\operatorname{\textsf{c}}}(0)$ such that $U_0(x_1,\ldots,x_j)\subseteq U_0$. The definition of the semilattice $\exp_n\lambda$ implies that for any $x_1,\ldots,x_j\in\exp_k\lambda$ the subsemilattices ${\uparrow}x_1,\ldots,{\uparrow}x_j$ of $\exp_k\lambda$ are isomorphic to the semilattices $\exp_{i_1}\lambda,\ldots,\exp_{i_j}\lambda$, respectively, for some non-negative integers $i_1,\ldots,i_j<k$. This, the definition of the topology $\tau_{\operatorname{\textsf{c}}}^{k}$ and the assumption of induction imply that ${\uparrow}x_1,\ldots,{\uparrow}x_j$ are compact subsets of $\left(\exp_{k}\lambda,\tau_{\operatorname{\textsf{c}}}^{k}\right)$. Then there exist finitely many $U_1,\ldots, U_m\in\mathscr{U}$ such that ${\uparrow}x_1\cup\cdots\cup{\uparrow}x_j\subseteq U_1\cup\cdots\cup U_m$, and hence $\left\{U_0,U_1,\ldots, U_m\right\}\subseteq\mathscr{U}$ is a finite subcover of the topological space $\left(\exp_{k}\lambda,\tau_{\operatorname{\textsf{c}}}^{k}\right)$.

Since in a Hausdorff compact semitopological semilattice the semilattice operation is continuous (see \cite[Proposition~VI-1.13]{Gierz-Hofmann-Keimel-Lawson-Mislove-Scott-2003} or \cite[p.~242, Theorem~6.6]{JLawson-1973}), it is sufficient to show that the semilattice operation in $\left(\exp_n\lambda,\tau_{\operatorname{\textsf{c}}}^n\right)$ is separately continuous.

Let $a$ and $b$ are arbitrary elements of the semilattice $\exp_n\lambda$. We consider the following three cases:
\begin{equation*}
  \hbox{\textbf{(I)}}~a=b; \qquad \hbox{\textbf{(II)}}~a<b; \qquad \hbox{and} \qquad \hbox{\textbf{(III)}}~a \hbox{~and~} b \hbox{~are~incomparable}.
\end{equation*}

In case \textbf{(I)} we have that $a\cdot U_a(x_1,\ldots,x_k)=\{a\}\subseteq U_a(x_1,\ldots,x_k)$ for any $U_a(x_1,\ldots,x_k)\in \mathscr{B}^n_{\operatorname{\textsf{c}}}(a)$.

In case \textbf{(II)} we get that $a\cdot U_b(b_1,\ldots,b_l)=\{a\}\subseteq U_a(x_1,\ldots,x_k)$ and $U_a(x_1,\ldots,x_k)\cdot b\subseteq U_a(x_1,\ldots,x_k)$ for any $U_a(x_1,\ldots,x_k)\in \mathscr{B}^n_{\operatorname{\textsf{c}}}(a)$ and $U_b(b_1,\ldots,b_l)\in\mathscr{B}^n_{\operatorname{\textsf{c}}}(b)$, because if $a\subseteq x\subseteq y$ and $a\subseteq b$ in $\exp_n\lambda$, then $a\subseteq x\cap b\subseteq y$.

In case \textbf{(III)} we consider two possible subcases: ${\uparrow}a\cap {\uparrow}b=\varnothing$ and ${\uparrow}a\cap {\uparrow}b\neq\varnothing$. Put $d=ab=a\cap b$. If ${\uparrow}a\cap {\uparrow}b=\varnothing$ then $a\cdot U_b(b_1,\ldots,b_l)=\{d\}\subseteq U_d(z_1,\ldots,z_k)$ and $U_a(x_1,\ldots,x_k)\cdot b\subseteq U_d(z_1,\ldots,z_k)$ for any $U_a(x_1,\ldots,x_k)\in \mathscr{B}^n_{\operatorname{\textsf{c}}}(a)$, $U_b(b_1,\ldots,b_l)\in\mathscr{B}^n_{\operatorname{\textsf{c}}}(b)$ and $U_d(z_1,\ldots,z_k)\in\mathscr{B}^n_{\operatorname{\textsf{c}}}(d)$, because in this subcase we have that ${\uparrow}a\cdot{\uparrow}b=d$. If ${\uparrow}a\cap {\uparrow}b\neq\varnothing$ then similar arguments as in the above case imply that $a\cdot U_b(b_1,\ldots,b_l,u)=\{d\}\subseteq U_d(z_1,\ldots,z_k)$ and $U_a(x_1,\ldots,x_k,u)\cdot b\subseteq U_d(z_1,\ldots,z_k)$ for any $U_a(x_1,\ldots,x_k,u)\in \mathscr{B}^n_{\operatorname{\textsf{c}}}(a)$, $U_b(b_1,\ldots,b_l,u)\in\mathscr{B}^n_{\operatorname{\textsf{c}}}(b)$ and $U_d(z_1,\ldots,z_k)\in\mathscr{B}^n_{\operatorname{\textsf{c}}}(d)$, where $u=a\cup b$ in $\exp_n\lambda$.

This completes the proof of our proposition.
\end{proof}

\begin{remark}
By Proposition~\ref{proposition-2.1}$(v)$ the topological space \emph{$\left(\exp_n\lambda,\tau_{\operatorname{\textsf{c}}}^n\right)$} is scattered. Since every countably compact scattered $T_3$-space is sequentially compact (see \cite[Theorem~5.7]{Vaughan-1984}), \emph{$\left(\exp_n\lambda,\tau_{\operatorname{\textsf{c}}}^n\right)$} is a sequentially compact space.
\end{remark}

\begin{theorem}\label{theorem-2.6}
Let $n$ be an arbitrary positive integer and $\lambda$ be an arbitrary infinite cardinal. Then for any $T_1$-topology $\tau$ on $\exp_n\lambda$ the following conditions are equivalent:
\begin{itemize}
  \item[$(i)$] $\left(\exp_n\lambda,\tau\right)$ is a compact topological semilattice;
  \item[$(ii)$] \emph{$\tau=\tau_{\operatorname{\textsf{c}}}^n$};
  \item[$(iii)$] $\left(\exp_n\lambda,\tau\right)$ is a countably compact topological semilattice;
  \item[$(iv)$] $\left(\exp_n\lambda,\tau\right)$ is a feebly compact topological semilattice;
  \item[$(v)$] $\left(\exp_n\lambda,\tau\right)$ is a compact semitopological semilattice;
  \item[$(vi)$] $\left(\exp_n\lambda,\tau\right)$ is a countably compact semitopological semilattice.
\end{itemize}
\end{theorem}

\begin{proof}
By Proposition~\ref{proposition-2.1} without loss of generality we may assume that $\tau$ is a Hausdorff topology on $\exp_n\lambda$. It is obvious that the following implications  $(i)\Rightarrow(iii)$, $(iii)\Rightarrow(iv)$, $(iii)\Rightarrow(vi)$, $(i)\Rightarrow(v)$ and $(v)\Rightarrow(vi)$ are trivial, and implication $(ii)\Rightarrow(i)$ follows from Proposition~\ref{proposition-2.5}.

$(i)\Rightarrow(ii)$. Suppose that $\tau$ is a compact topology on $\exp_n\lambda$ such that $\left(\exp_n\lambda,\tau\right)$ is a topological semilattice. Then by Proposition~\ref{proposition-2.1}$(iii)$ the identity map 
 \begin{equation*}
\operatorname{\textsf{id}}_{\exp_n\lambda}\colon \left(\exp_n\lambda,\tau\right)\rightarrow \left(\exp_n\lambda,\tau_{\operatorname{\textsf{c}}}^n\right)
\end{equation*}
is continuous, and hence by Theorem~2.1.13 of \cite{Engelking-1989} is a homeomorphism. Thus, we get that $\tau=\tau_{\operatorname{\textsf{c}}}^n$.

Implication $(v)\Rightarrow(i)$ follows from Proposition~VI-1.13 of~\cite{Gierz-Hofmann-Keimel-Lawson-Mislove-Scott-2003} (also from Theorem~6.6 of \cite[p.~242]{JLawson-1973}).

$(vi)\Rightarrow(v)$. We shall prove this implication by induction.

Assume that $n=1$. Suppose to the contrary that there exists a non-compact topology $\tau$ on $\exp_1\lambda$ such that $\left(\exp_1\lambda,\tau\right)$ is a countably compact semitopological semilattice. Then there exists an open cover $\mathscr{U}$ of the space $\left(\exp_1\lambda,\tau\right)$ which contans no a finite subcover. This implies that there exists $U\in\mathscr{U}$ such that $0\in U$ and $\exp_1\lambda\setminus U$ is infinite subset of $\exp_1\lambda$. Then by Proposition~\ref{proposition-2.1}$(iii)$ the space $\left(\exp_1\lambda,\tau\right)$ contains an open-and-closed discrete subspace, which contradicts Theorem~3.10.3 from \cite{Engelking-1989}. Thus, $\left(\exp_1\lambda,\tau\right)$ is a compact semitopological semilattice.

Next, we shall prove that the statement the countably compact semitopological semilattice $\left(\exp_i\lambda,\tau\right)$ is compact for all positive integers $i<k\leqslant n$ implies that the countably compact semitopological semilattice $\left(\exp_{k}\lambda,\tau\right)$ is compact too. Then there exists an open cover $\mathscr{U}$ of the topological space $\left(\exp_{k}\lambda,\tau\right)$ which contains no a finite subcover. Then by Proposition~\ref{proposition-2.1}$(i)$, $\exp_{k-1}\lambda$ is a closed subset of $\left(\exp_{k}\lambda,\tau\right)$, and hence by Theorem~3.10.4 of~\cite{Engelking-1989}, $\exp_{k-1}\lambda$ is countably compact. The assumption of induction implies that $\exp_{k-1}\lambda$ is a compact subspace of $\left(\exp_{k}\lambda,\tau\right)$, and hence the open cover $\mathscr{U}$ of the topological space $\left(\exp_{k}\lambda,\tau\right)$ contains a finite subcover $\mathscr{U}_0$ of $\exp_{k-1}\lambda$. If the open cover $\mathscr{U}$ of the topological space $\left(\exp_{k}\lambda,\tau\right)$ contains no a finite subcover of $\left(\exp_{k}\lambda,\tau\right)$ then by Proposition~\ref{proposition-2.1}$(iii)$ we have that $\exp_{k}\lambda\setminus\bigcup\mathscr{U}_0$ is an open-and-closed discrete subspace of $\left(\exp_{k}\lambda,\tau\right)$, which contradicts Theorem~3.10.3 from \cite{Engelking-1989}. Thus, $\left(\exp_k\lambda,\tau\right)$ is a compact semitopological semilattice. This completes the proof of our implication.

$(iv)\Rightarrow(iii)$. We shall prove this implication by induction.

Assume that $n=1$. Suppose to the contrary that there exists a feebly compact topological semilattice $\tau$ on $\exp_1\lambda$ such that $\left(\exp_1\lambda,\tau\right)$ is not a countably compact space. Then there exists a countable open cover $\mathscr{U}$ of the space $\left(\exp_1\lambda,\tau\right)$ which contains no a finite subcover. This implies that there exists $U\in\mathscr{U}$ such that $0\in U$ and $\exp_1\lambda\setminus U$ is infinite subset of $\exp_1\lambda$. Then by Proposition~\ref{proposition-2.1}$(iii)$ the space $\left(\exp_1\lambda,\tau\right)$ contains an open-and-closed discrete subspace, which contradicts the feeble compactness of $\left(\exp_1\lambda,\tau\right)$, a contradiction. Hence $\left(\exp_1\lambda,\tau\right)$ is a countably compact space.

Next, we shall prove that the statement that every feebly compact topological semilattice $\left(\exp_i\lambda,\tau\right)$ is countably compact for all positive integers $i<k\leqslant n$ implies that the feebly compact topological semilattice $\left(\exp_{k}\lambda,\tau\right)$ is countably compact too.

Suppose to the contrary that every feebly compact topological semilattice $\left(\exp_i\lambda,\tau\right)$ is countably compact for all positive integers $i<k\leqslant n$ but there exists a feebly compact topological semilattice $\left(\exp_{k}\lambda,\tau\right)$ which is not countably compact. Then by Theorem~3.10.3 from \cite{Engelking-1989} the topological semilattice $\left(\exp_{k}\lambda,\tau\right)$ contains an infinite closed discrete subspace $A$. Since by Proposition~\ref{proposition-2.1}$(ii)$, $\exp_{k}\lambda\setminus\exp_{k-1}\lambda$ is an open discrete subspace of $\left(\exp_{k}\lambda,\tau\right)$, the feeble compactness of $\left(\exp_{k}\lambda,\tau\right)$ implies that $A\subseteq \exp_{k-1}\lambda$. Also, by Proposition~\ref{proposition-2.1}$(iii)$ since ${\uparrow}x$ is an open-and-closed subset of the space $\left(\exp_{k}\lambda,\tau\right)$ for every $x\in \exp_{k}\lambda$ we have that ${\uparrow}x$ is a feebly compact subspace of $\left(\exp_{k}\lambda,\tau\right)$. It is obvious that for any non-zero element $x\in\exp_{k}\lambda$ the subsemilattice ${\uparrow}x$ of $\exp_{k}\lambda$ is isomorphic to semilattice $\exp_{m}\lambda$ for some non-negative integer $m<k$. This and the assumption of induction imply that ${\uparrow}x$ is a countably compact subspace of $\left(\exp_{k}\lambda,\tau\right)$ for any non-zero element $x$ of the semilattice $\exp_{k}\lambda$. Hence we get that the set $A\cap {\uparrow}x$ is finite for any non-zero element $x$ of the semilattice $\exp_{k}\lambda$.

Next, by induction we shall show that if for some positive integer $i$ with $2\leqslant i<n$ in a feebly compact topological semilattice $\left(\exp_n\lambda,\tau\right)$ there exists an open neighbourhood $U(0)$ of zero $0$ in $\left(\exp_n\lambda,\tau\right)$ such that $U(0)$ does not contain an infinite subset $A$ of $\exp_i\lambda\setminus \exp_{i-1}\lambda$ such that $A\cap {\uparrow}x$ is finite for any non-zero element $x\in\exp_i\lambda$ and ${\uparrow}x$ is countably compact, then there exists an open neighbourhood $V(0)\subseteq U(0)$ of zero $0$ in $\left(\exp_n\lambda,\tau\right)$ such that $V(0)$ does not contain an infinite subset $A_+$ of $\exp_{i+1}\lambda\setminus \exp_{i}\lambda$ such that $A_+\cap {\uparrow}x$ is finite for any non-zero element $x\in\exp_i\lambda$.

Suppose that in a feebly compact topological semilattice $\left(\exp_n\lambda,\tau\right)$ there exists an open neighbourhood $U(0)$ of zero $0$ such that $U(0)\cap A=\varnothing$ for some infinite subset $A=\left\{x_i\colon i\in\mathbb{N}\right\}\subseteq\exp_1\lambda\setminus\{0\}$. Then the continuity of the semilattice operation in $\left(\exp_n\lambda,\tau\right)$ implies that there exists an open neighbourhood $V(0)\subseteq U(0)$ of zero in $\left(\exp_n\lambda,\tau\right)$ such that $V(0)\cdot V(0)\subseteq U(0)$. Suppose that there exist some distinct $x_{i_0},x_{i_1}\in A$ such that $\left\{x_{i_0},x_{i_1}\right\}\in V(0)$. Then by the inclusion $V(0)\cdot V(0)\subseteq U(0)$ we have that
\begin{equation*}
  \left\{\left\{x_{i_0},x_{i}\right\}\colon i\in\mathbb{N}\setminus\left\{{i_0},{i_1}\right\}\right\}\cap V(0)=\varnothing.
\end{equation*}
This implies that the subspace ${\uparrow}\{x_{i_0}\}$ of $\left(\exp_n\lambda,\tau\right)$ contains a closed discrete subspace, which contradicts the countable compactness of ${\uparrow}\{x_{i_0}\}$. Hence we get that $V(0)\cap A_+=\varnothing$, where
\begin{equation*}
A_+=\left\{\left\{x_{k},x_{l}\right\}\colon x_{k},x_{l} \hbox{~are distinct elements of~} A\right\}.
\end{equation*}

Suppose that in a feebly compact topological semilattice $\left(\exp_n\lambda,\tau\right)$ there exist an open neighbourhood $U(0)$ of zero $0$ and infinite subset $A\subseteq \exp_n\lambda$ such that $U(0)\cap A=\varnothing$ and $|x|=j>1$ for any $x\in A$. Since for any non-zero element $a\in\exp_n\lambda$ the subspace ${\uparrow}a$ is countably compact, without loss of generality we may assume that there exists a countable set $A_1=\left\{x_i\colon i\in\mathbb{N}\right\}$ which consists of singletons from $\exp_n\lambda$ such that $A\cap{\uparrow}x_i$ is a singleton for any positive integer $i$. Then the continuity of the semilattice operation in $\left(\exp_n\lambda,\tau\right)$ implies that there exists an open neighbourhood $V(0)\subseteq U(0)$ of zero in $\left(\exp_n\lambda,\tau\right)$ such that $V(0)\cdot V(0)\subseteq U(0)$. We claim that for any distinct elements $x_p,x_s\in A_1$, $s,p\in\mathbb{N}$ there exists no $x\in{\uparrow}x_p$ such that $y=\left\{\left\{x_s\right\}\cup x\right\}\notin V(0)$. Indeed, in the other case the neighbourhood $V(0)$ does not contain the set $\left\{\left\{x_q\right\}\cup x\colon x_q\in A\setminus \left\{x_s\right\}\right\}$. This implies that the subspace ${\uparrow}x_p$ of  $\left(\exp_n\lambda,\tau\right)$ contains an infinite closed discrete subspace, which contradicts the assumption that ${\uparrow}x_p$ is a countably compact subspace of $\left(\exp_n\lambda,\tau\right)$.
Hence we get that $V(0)\cap A_+=\varnothing$, where
\begin{equation*}
A_+=\left\{\left\{x_i\right\}\cup x\colon x_{i}\in A_1 \hbox{~and~} x\in A\right\}.
\end{equation*}

The above presented arguments imply that the topological semilattice $\left(\exp_n\lambda,\tau\right)$ contains an infinite open-and-closed discrete subspace, which contradicts the feeble compactness of the space  $\left(\exp_n\lambda,\tau\right)$. The obtained contradiction implies the requested implication.
\end{proof}

Proposition~\ref{proposition-2.1}$(iii)$ implies the following corollary.

\begin{corollary}\label{corollary-2.7}
Let $\lambda$ be an arbitrary infinite cardinal. Then every feebly compact $T_1$-topology $\tau$ on the semilattice $\exp_1\lambda$ such that $\left(\exp_1\lambda,\tau\right)$ is a semitopological semilattice, is compact, and hence $\left(\exp_1\lambda,\tau\right)$ is a topological semilattice.
\end{corollary}

But, the following example shows that for any infinite cardinal $\lambda$ and any positive integer $n\geqslant 2$ there exists a Hausdorff feebly compact topology $\tau$ on the semilattice $\exp_n\lambda$ such that $\left(\exp_n\lambda,\tau\right)$ is a non-countably compact semitopological semilattice.

\begin{example}\label{example-2.8}
Let $\lambda$ be any infinite cardinal and $\tau_{\operatorname{\textsf{c}}}^2$ be the topology on the semilattice $\exp_2\lambda$ defined in Example~\ref{example-2.4}. We construct more stronger topology $\tau_{\operatorname{\textsf{fc}}}^2$ on $\exp_2\lambda$ them $\tau_{\operatorname{\textsf{c}}}^2$ in the following way. By $\pi\colon\lambda\to\exp_2\lambda\colon a\mapsto\left\{a\right\}$ we denote the natural embedding of $\lambda$ into $\exp_2\lambda$. Fix an arbitrary infinite subset $A\subseteq\lambda$ of cardinality $\leqslant\lambda$. For every non-zero element $x\in\exp_2\lambda$ we put the base $\mathscr{B}_{\operatorname{\textsf{fc}}}^2(x)$ of the topology $\tau_{\operatorname{\textsf{fc}}}^2$ at the point $x$ coincides with the base of the topology $\tau_{\operatorname{\textsf{c}}}^2$ at $x$, and
\begin{equation*}
  \mathscr{B}_{\operatorname{\textsf{fc}}}^2(0)=\left\{U_B(0)=U(0)\setminus \pi(B)\colon U(0)\in\mathscr{B}_{\operatorname{\textsf{c}}}^2(0), B\subseteq\lambda
\hbox{~and the set~} A\setminus B\cup B\setminus A \hbox{~is finite}\right\}
\end{equation*}
form a base of the topology $\tau_{\operatorname{\textsf{fc}}}^2$ at zero $0$ of the semilattice $\exp_2\lambda$. Simple verifications show that the family $\left\{\mathscr{B}_{\operatorname{\textsf{fc}}}^2(x)\colon x\in\exp_2\lambda\right\}$ satisfies the conditions \textbf{(BP1)--(BP4)} of \cite{Engelking-1989}, and hence $\tau_{\operatorname{\textsf{fc}}}^2$ is a Hausdorff topology on $\exp_2\lambda$.
\end{example}

\begin{proposition}\label{proposition-2.9}
Let $\lambda$ be an arbitrary infinite cardinal. Then \emph{$\left(\exp_2\lambda,\tau_{\operatorname{\textsf{fc}}}^2\right)$} is a countably pracompact semitopological semilattice such that \emph{$\left(\exp_2\lambda,\tau_{\operatorname{\textsf{fc}}}^2\right)$} is an $H$-closed non-semiregular space.
\end{proposition}

\begin{proof}
The definition of the topology $\tau_{\operatorname{\textsf{fc}}}^2$ implies that it is sufficient to show that the semilattice operation is separately continuous in the case $x\cdot 0$. Fix an arbitrary basic neighbourhood $U_B(0)$ of zero in $\left(\exp_2\lambda,\tau_{\operatorname{\textsf{fc}}}^2\right)$. If $x$ is a singleton of $\lambda$, i.e., $x=\{x_0\}$ in $\exp_2\lambda$, then we have that $x\cdot V_B(0)=\{0\}\subseteq U_B(0)$, where $V(0)=U(0)\setminus{\uparrow}x$. In the case when $x$ is a two-elements subset of $\lambda$, where $x=\{x_1,x_2\}$ for some $x_1,x_2\in\lambda$, the we get that $x\cdot W_B(0)=\{0\}\subseteq U_b(0)$, where $W(0)=U(0)\setminus\left({\uparrow}\left\{x_1\right\}\cup{\uparrow}\left\{x_2\right\}\right)$.

Also, the definition of the topology $\tau_{\operatorname{\textsf{fc}}}^2$ on $\exp_2\lambda$ implies that the set $\exp_2\lambda\setminus\exp_1\lambda$ is dense in $\left(\exp_2\lambda,\tau_{\operatorname{\textsf{fc}}}^2\right)$ and every infinite subset of $\exp_2\lambda\setminus\exp_1\lambda$ has an accumulation point in $\left(\exp_2\lambda,\tau_{\operatorname{\textsf{fc}}}^2\right)$, and hence the space $\left(\exp_2\lambda,\tau_{\operatorname{\textsf{fc}}}^2\right)$ is countably pracompact.

Suppose to the contrary that $\left(\exp_2\lambda,\tau_{\operatorname{\textsf{fc}}}^2\right)$ is not an $H$-closed topological space. Then there exists a Hausdorff topological space $X$ which contains $\left(\exp_2\lambda,\tau_{\operatorname{\textsf{fc}}}^2\right)$ as a dense proper subspace. Fix an arbitrary $x\in X\setminus \exp_2\lambda$. Since $X$ is Hausdorff there exist disjunctive open neighbourhoods $U(x)$ and $U(0)$ of $x$ and zero $0$ of the semilattice $\exp_2\lambda$ in $X$, respectively. Then there exists a basic neighbourhood $V_B(0)$ of zero in $\left(\exp_2\lambda,\tau_{\operatorname{\textsf{fc}}}^2\right)$ such that $V_B(0)\subseteq \exp_2\lambda\cap U(0)$. Also, the definition of the base $\mathscr{B}_{\operatorname{\textsf{fc}}}^2(0)$ of the topology $\tau_{\operatorname{\textsf{fc}}}^2$ at zero $0$ of the semilattice $\exp_2\lambda$ implies that there exist finitely many non-zero elements $x_1,\ldots,x_m$ of the semilattice $\exp_2\lambda$ such that $\exp_2\lambda\setminus\left({\uparrow}x_1\cup\ldots\cup{\uparrow}x_m\cup V_B(0)\right)\subseteq B$, and since by Proposition~\ref{proposition-2.1}$(iii)$ the subsets ${\uparrow}x_1,\ldots,{\uparrow}x_m$ are open-and-closed in $\left(\exp_2\lambda,\tau_{\operatorname{\textsf{fc}}}^2\right)$ without loss of generality we may assume that $U(x)\cap\exp_2\lambda\subseteq B$. If the set $U(x)\cap\exp_2\lambda\subseteq B$ is infinite then the space $\left(\exp_2\lambda,\tau_{\operatorname{\textsf{fc}}}^2\right)$ contains a discrete infinite open-and-closed subspace, which contradicts the feeble compactness of $\left(\exp_2\lambda,\tau_{\operatorname{\textsf{fc}}}^2\right)$. The obtained contradiction implies that the space $\left(\exp_2\lambda,\tau_{\operatorname{\textsf{fc}}}^2\right)$ is $H$-closed.
\end{proof}

\begin{remark}\label{remark-2.10}
If $n$ is an arbitrary positive integer $\geqslant 3$, $\lambda$ is any infinite cardinal and $\tau_{\operatorname{\textsf{c}}}^n$ is the topology on the semilattice $\exp_n\lambda$ defined in Example~\ref{example-2.4}, then we construct more stronger topology $\tau_{\operatorname{\textsf{fc}}}^n$ on $\exp_n\lambda$ them $\tau_{\operatorname{\textsf{c}}}^2$ in the following way. Fix an arbitrary element $x\in \exp_n\lambda$ such that $|x|=n-2$. It is easy to see that the subsemilattice ${\uparrow}x$ of $\exp_n\lambda$ is isomorphic to $\exp_2\lambda$, and by $h\colon\exp_2\lambda\to{\uparrow}x$ we denote this isomorphism.

Fix an arbitrary subset $A\subseteq\lambda$ of cardinality $\leqslant\lambda$. For every zero element $y\in\exp_n\lambda\setminus{\uparrow}x$ we put the base $\mathscr{B}_{\operatorname{\textsf{fc}}}^n(y)$ of the topology $\tau_{\operatorname{\textsf{fc}}}^n$ at the point $y$ coincides with the base of the topology $\tau_{\operatorname{\textsf{c}}}^n$ at $y$, and put ${\uparrow}x$ is an open-and-closed subset and the topology on ${\uparrow}x$ is generated by map $h\colon\left(\exp_2\lambda,\tau_{\operatorname{\textsf{fc}}}^2\right)\to{\uparrow}x$. Simple verifications as in the proof of Proposition~\ref{proposition-2.9} show that \emph{$\left(\exp_n\lambda,\tau_{\operatorname{\textsf{fc}}}^n\right)$} is a countably pracompact semitopological semilattice such that \emph{$\left(\exp_n\lambda,\tau_{\operatorname{\textsf{fc}}}^n\right)$} is an $H$-closed quasiregular non-semiregular space.
\end{remark}

\begin{remark}\label{remark-2.10a}
Simple verifications show that \emph{$\left(\exp_2\lambda,\tau_{\operatorname{\textsf{fc}}}^2\right)$} is not a $0$-dimensional space. This implies that the term ``hereditarily disconnected'' in item $(v)$ of Proposition~\ref{proposition-2.1} can not be replaced by ``$0$-dimensional''.
\end{remark}

A $T_1$-space $X$ is called \emph{collectionwise normal} if for every discrete family $\left\{F_s\right\}_{s\in\mathscr{A}}$ of closed subsets of $X$ there exists a discrete family $\left\{U_s\right\}_{s\in\mathscr{A}}$ of open subsets of $X$ such that $F_s\subseteq U_s$ for every $s\in\mathscr{A}$ \cite{Engelking-1989}.

\begin{proposition}\label{proposition-2.11}
Let $\lambda$ be an arbitrary infinite cardinal and $\tau$ be a $T_1$-topology on $\exp_1\lambda$ such that $\left(\exp_1\lambda,\tau\right)$ is a semitopological semilattice. Then the space $\left(\exp_1\lambda,\tau\right)$ is collectionwise normal.
\end{proposition}

\begin{proof}
Suppose that $\left\{F_s\right\}_{s\in\mathscr{A}}$ is a discrete family of closed subsets of $\left(\exp_1\lambda,\tau\right)$. By Proposition~\ref{proposition-2.1}$(iii)$ all non-zero elements of the semilattice $\exp_1\lambda$ are isolated points in the space $\left(\exp_1\lambda,\tau\right)$. Hence, if there exists an open neighbourhood $U(0)$ of zero in $\left(\exp_1\lambda,\tau\right)$ such that $U(0)\cap F_{s}=\varnothing$ for all $s\in\mathscr{A}$ then we put $U_s=F_s$ for all $s\in\mathscr{A}$. In other case there exists an open neighbourhood $U(0)$ of zero in $\left(\exp_1\lambda,\tau\right)$ such that $U(0)\cap F_{s_0}\neq\varnothing$ for some $s_0\in\mathscr{A}$ and $U(0)\cap F_{s}=\varnothing$ for all $s\in\mathscr{A}\setminus\left\{s_0\right\}$. We put
\begin{equation*}
  U_s=
  \left\{
    \begin{array}{cl}
      F_s,              & \hbox{if~} s\in\mathscr{A}\setminus\left\{s_0\right\}; \\
      F_{s_0}\cup U(0), & \hbox{if~} s=s_0.
    \end{array}
  \right.
\end{equation*}
Then Proposition~\ref{proposition-2.1}$(ii)$ implies that $\left\{U_s\right\}_{s\in\mathscr{A}}$ is a discrete family $\left\{U_s\right\}_{s\in\mathscr{A}}$ of open subsets of $\left(\exp_1\lambda,\tau\right)$ such that $F_s\subseteq U_s$ for every $s\in\mathscr{A}$, and hence the space $\left(\exp_1\lambda,\tau\right)$ is collectionwise normal.
\end{proof}

\begin{remark}\label{remark-2.12}
A topological space $X$ is called \emph{perfectly normal} if $X$ is normal and every closed subset of $X$ is a $G_\delta$-set. It is obvious that if $\lambda$ is any uncountable cardinal then $\left(\exp_1\lambda,\tau_{\operatorname{\textsf{c}}}^1\right)$ is a compact space which is not perfectly normal (see: \cite[Section~1.5]{Engelking-1989}).
\end{remark}

\begin{theorem}\label{theorem-2.13}
Let $n$ be an arbitrary positive integer and $\lambda$ be an arbitrary infinite cardinal. Then every semilattice $T_1$-topology on $\exp_n\lambda$ is regular.
\end{theorem}

\begin{proof}
Suppose that $\tau$ is a $T_1$-topology on $\exp_n\lambda$ such that $\left(\exp_n\lambda,\tau\right)$ is a topological semilattice. In the case when $n=1$ the statement of the theorem follows from Proposition~\ref{proposition-2.11}. Hence, later we assume that $n\geqslant 2$.

By Proposition~\ref{proposition-2.1}$(iii)$, ${\uparrow}x$ is an open-and-closed subsemilattice of $\left(\exp_n\lambda,\tau\right)$ for any $x\in\exp_n\lambda$, and hence it will be sufficient to show that for every open neighbourhood $U(0)$ of zero in $\left(\exp_n\lambda,\tau\right)$ there exists an open neighbourhood $V(0)$ of zero in $\left(\exp_n\lambda,\tau\right)$ such that $\operatorname{cl}_{\exp_n\lambda}(V(0))\subseteq U(0)$.

Fix an arbitrary open neighbourhood $U(0)$ of zero in $\left(\exp_n\lambda,\tau\right)$. Then the continuity of the semilattice operation in $\left(\exp_n\lambda,\tau\right)$ implies that there exists an open neighbourhood $V(0)\subseteq U(0)$ of zero in $\left(\exp_n\lambda,\tau\right)$ such that $V(0)\cdot V(0)\subseteq U(0)$. Suppose that there exists $x\in\operatorname{cl}_{\exp_n\lambda}(V(0))\setminus V(0)$. By Proposition~\ref{proposition-2.1}$(ii)$ we have that $x\in{\downarrow}V(0)$. We assume that $x=\left\{a_1,\ldots,a_i\right\}$ as a finite subset of the cardinal $\lambda$, where $i<n$, i.e., $x\in\exp_i\lambda\setminus\exp_{i-1}\lambda$. Then $V(x)\cap V(0)\neq\varnothing$ for every open neighbourhood $V(x)$ of the point $x$ in $\left(\exp_n\lambda,\tau\right)$. Proposition~\ref{proposition-2.1}$(iii)$ implies that without loss of generality we may assume that $V(x)\subseteq{\uparrow}x$.  Fix an arbitrary $y\in\left(V(0)\cap V(x)\right)\setminus\{x\}$. Then we may assume that $y=\left\{a_1,\ldots,a_i,a_{i+1},\ldots,a_j\right\}$ as a finite subset of the cardinal $\lambda$, where $i<j\leqslant n$. We put
\begin{equation*}
  x_1=\left\{a_1,\ldots,a_i,a_{i+1}\right\},  \ldots , x_{j-i}=\left\{a_1,\ldots,a_i,a_{j}\right\},
\end{equation*}
as finite subsets of the cardinal $\lambda$. Then the semilattice operation of $\exp_n\lambda$ implies that $y\in{\uparrow}x_1\cup\cdots\cup{\uparrow}x_{j-i} \subseteq {\uparrow}x$ and $y\cdot z=x$ for every $z\in{\uparrow}x\setminus\left({\uparrow}x_1\cup\cdots\cup{\uparrow}x_{j-i}\right)$. Since $x\in\operatorname{cl}_{\exp_n\lambda}(V(0))\setminus V(0)$, Proposition~\ref{proposition-2.1}$(iii)$ implies that $W(x)=V(x)\setminus\left({\uparrow}x_1\cup\cdots\cup{\uparrow}x_{j-i}\right)$ is an open neighbourhood of the point $x$ in $\left(\exp_n\lambda,\tau\right)$. Then the above arguments imply that $x=y\cdot W(x)\subseteq V(0)\cdot V(0)\subseteq U(0)$ and hence $\operatorname{cl}_{\exp_n\lambda}(V(0))\subseteq U(0)$.
\end{proof}

Since in any countable $T_1$-space $X$ every open subset of $X$ is a $F_\sigma$-set, Theorem~1.5.17 from \cite{Engelking-1989} and Theorem~\ref{theorem-2.13} imply the following corollary.

\begin{corollary}\label{corollary-2.14}
Let $n$ be an arbitrary positive integer. Then every semilattice $T_1$-topology on $\exp_n\omega$ is perfectly normal.
\end{corollary}

Later we need the following lemma:

\begin{lemma}\label{lemma-2.15}
Let $n$ be an arbitrary positive integer and $\lambda$ be an arbitrary infinite cardinal. Let $\left(\exp_n\lambda,\tau\right)$ be a Hausdorff feebly compact semitopological semilattice. Then for every open neighbourhood $U(0)$ of zero in $\left(\exp_n\lambda,\tau\right)$ there exist finitely many non-zero elements $x_1,\ldots,x_i\in \exp_n\lambda$ such that
\begin{equation*}
\exp_n\lambda\setminus \exp_{n-1}\lambda\subseteq U(0)\cup {\uparrow}x_1\cup\ldots\cup {\uparrow}x_i.
\end{equation*}
\end{lemma}

\begin{proof}
By Proposition~\ref{proposition-2.1}$(ii)$ every point $x\in \exp_n\lambda\setminus\exp_{n-1}\lambda$ is isolated in $\left(\exp_n\lambda,\tau\right)$. Next, we apply the feeble compactness of $\left(\exp_n\lambda,\tau\right)$ and Proposition~\ref{proposition-2.1}$(iii)$.
\end{proof}

\begin{theorem}\label{theorem-2.16}
Let $n$ be an arbitrary positive integer and $\lambda$ be an arbitrary infinite cardinal. Then every semiregular feebly compact $T_1$-topology $\tau$ on $\exp_n\lambda$ such that $\left(\exp_n\lambda,\tau\right)$ is a semitopological semilattice, is compact, and hence the semilattice operation in $\left(\exp_n\lambda,\tau\right)$ is continuous.
\end{theorem}

\begin{proof}
We shall prove the statement of the theorem by induction. In the case when $n=1$ the statement of the theorem follows from Corollary~\ref{corollary-2.7}. First we consider the initial step: $n=2$. Suppose to the contrary that there exists a semiregular feebly compact non-compact $T_1$-semitopological semilattice $\left(\exp_2\lambda,\tau\right)$. By Theorem~\ref{theorem-2.6} the topological space $\left(\exp_2\lambda,\tau\right)$ is not countably compact, and hence Theorem~3.10.3 of \cite{Engelking-1989} implies that $\left(\exp_2\lambda,\tau\right)$ contains an infinite closed discrete subspace $X$. Now, Proposition~\ref{proposition-2.1}$(iii)$ implies that $\exp_2\lambda\setminus\exp_1\lambda$ is an open discrete subspace of $\left(\exp_2\lambda,\tau\right)$, and since $\left(\exp_2\lambda,\tau\right)$ is feebly compact, without loss of generality we may assume that $X\subseteq \exp_1\lambda\setminus\{0\}$. Fix an arbitrary regular open neighbourhood $U(0)$ of zero in $\left(\exp_2\lambda,\tau\right)$ such that $U(0)\cap X=\varnothing$.

For every $x\in \exp_{1}\lambda\setminus\{0\}$ the subset ${\uparrow}x$ is open-and-closed in $\left(\exp_2\lambda,\tau\right)$ and hence ${\uparrow}x$ is feebly compact. Since ${\uparrow}x$ is algebraically isomorphic to $\exp_{1}\lambda$, Corollary~\ref{corollary-2.7} implies that ${\uparrow}x$ is compact. By Lemma~\ref{lemma-2.15} there exist finitely many non-zero elements $x_1,\ldots,x_i\in \exp_2\lambda$ such that
\begin{equation*}
\exp_2\lambda\setminus \exp_{1}\lambda\subseteq U(0)\cup {\uparrow}x_1\cup\ldots\cup {\uparrow}x_i.
\end{equation*}
The semilattice operation of $\left(\exp_2\lambda,\tau\right)$ implies that without loss of generality we may assume that $x_1,\ldots,x_i$ are singleton subsets of the cardinal $\lambda$. This and above presented arguments imply that $\operatorname{cl}_{\exp_2\lambda}(U(0))\cap X\neq\varnothing$. Moreover, we have that the narrow $\operatorname{cl}_{\exp_2\lambda}(U(0))\setminus U(0)$ consists of singleton subsets of the cardinal $\lambda$. Then for every $x\in \operatorname{cl}_{\exp_2\lambda}(U(0))\setminus U(0)$ by Corollary~\ref{corollary-2.7}, ${\uparrow}x$ is a  compact topological subsemilattice of $\left(\exp_2\lambda,\tau\right)$. Now, Theorem~\ref{theorem-2.6} and Lemma~\ref{lemma-2.15} imply that $x\in \operatorname{int}_{\exp_2\lambda} \left(\operatorname{cl}_{\exp_2\lambda}(U(0))\right)=U(0)$, which contradicts the assumption $U(0)\cap X=\varnothing$. The obtained contradiction implies that $\left(\exp_2\lambda,\tau\right)$ is a compact semitopological semilattice.

Next we shall show the step of induction, i.e., that the statement \emph{if for every positive integer $l<n$ a semiregular feebly compact $T_1$-semitopological semilattice $\left(\exp_l\lambda,\tau\right)$ is compact} implies that a semiregular feebly compact $T_1$-semitopological semilattice $\left(\exp_n\lambda,\tau\right)$ is compact too. Suppose to the contrary that there exists a semiregular feebly compact non-compact $T_1$-semitopological semilattice $\left(\exp_n\lambda,\tau\right)$ which is not compact. By Theorem~\ref{theorem-2.6} the topological space $\left(\exp_n\lambda,\tau\right)$ is not countably compact, and hence Theorem~3.10.3 of \cite{Engelking-1989} implies that $\left(\exp_n\lambda,\tau\right)$ contains an infinite closed discrete subspace $X$. Now, by Proposition~\ref{proposition-2.1}$(iii)$, $\exp_n\lambda\setminus\exp_{n-1}\lambda$ is an open discrete subspace of $\left(\exp_n\lambda,\tau\right)$, and since $\left(\exp_n\lambda,\tau\right)$ is feebly compact, without loss of generality we may assume that $X\subseteq \exp_{n-1}\lambda\setminus\{0\}$.

Put $k<n$ is the maximum positive integer such that the set $\exp_k\lambda\setminus\exp_{k-1}\lambda\cap X$ is infinite. We observe that for any non-zero element $x\in\exp_n\lambda$ the subsemilattice ${\uparrow}x$ of $\exp_n\lambda$ is algebraically isomorphic to the semilattice $\exp_j\lambda$ for some positive integer $j<n$, and since by Proposition~\ref{proposition-2.1}$(iii)$, ${\uparrow}x$ is an open-and-closed subset of a feebly compact semitopological semilattice $\left(\exp_n\lambda,\tau\right)$, the assumption of induction implies that ${\uparrow}x$ is a compact subsemilattice of $\left(\exp_n\lambda,\tau\right)$. This implies that there do not exist finitely many non-zero elements $y_1,\ldots,y_s$ of the semitopological semilattice $\left(\exp_n\lambda,\tau\right)$ such that $X\subseteq{\uparrow}y_1\cup\ldots\cup{\uparrow}y_s$.

Fix an arbitrary regular open neighbourhood $U(0)$ of zero in $\left(\exp_n\lambda,\tau\right)$ such that $U(0)\cap X=\varnothing$. Then the above arguments imply that $\operatorname{cl}_{\exp_n\lambda}(V(0))\cap\left(\exp_k\lambda\cap X\right)\neq\varnothing$. Moreover, we have that the narrow $\operatorname{cl}_{\exp_n\lambda}(U(0))\setminus U(0)$ contains infinitely many $k$-element subsets of the cardinal $\lambda$ which belongs to the set $X$. Then for every such element $x\in \operatorname{cl}_{\exp_n\lambda}(U(0))\setminus U(0)$ the assumtion of induction implies that ${\uparrow}x$ is a compact topological subsemilattice of $\left(\exp_n\lambda,\tau\right)$. Now, Theorem~\ref{theorem-2.6} and Lemma~\ref{lemma-2.15} imply that every such element $x$ belongs to the set $\operatorname{int}_{\exp_n\lambda} \left(\operatorname{cl}_{\exp_n\lambda}(U(0))\right)=U(0)$, which contradicts the assumption $U(0)\cap X=\varnothing$. The obtained contradiction implies that $\left(\exp_n\lambda,\tau\right)$ is a compact semitopological semilattice.

The last assertion of the theorem follows from Theorem~\ref{theorem-2.6}.
\end{proof}

\section*{Acknowledgements}

We acknowledge Taras Banakh and the referee for they useful comments and
suggestions.

\end{document}